\documentclass[12pt,oneside]{amsart}

\usepackage{setspace,amssymb,amsmath,mathrsfs,enumitem,textcomp,amsbsy,url}
\usepackage[all,cmtip]{xy}

\usepackage{datetime}
\renewcommand{\dateseparator}{-}
\renewcommand{\today}{\the\year \dateseparator \twodigit\month
\dateseparator \twodigit\day}

\usepackage
[pdfauthor={Stergios M. Antonakoudis},
 bookmarks=false,
 hyperfootnotes=false]
{hyperref}

\usepackage{needspace,graphicx,extpfeil,pdfpages}

\usepackage{fancyhdr}
\fancypagestyle{plain}
{
  \fancyhf{}               
  \fancyhead[C]{\thepage}  
}

\newcounter{restatecount}
\newcommand{\restate}[3][Theorem]
{
\addtocounter{restatecount}{1}
\theoremstyle{plain}
\newtheorem*{thm-\arabic{restatecount}}{#1 \ref{#2}}
\begin{thm-\arabic{restatecount}}
#3
\end{thm-\arabic{restatecount}}
}

\theoremstyle{plain}
\newtheorem{thm}{Theorem}[section]

\newtheorem{thm-ref}{Theorem}

\newtheorem{lem}[thm]{Lemma}
\newtheorem{cor}[thm]{Corollary}
\newtheorem{prop}[thm]{Proposition}

\theoremstyle{definition}

\newtheorem*{examples}{Examples}

\theoremstyle{remark}
\newtheorem*{remark}{Remark}

\numberwithin{equation}{section}

\addtocontents{toc}{\protect\setcounter{tocdepth}{1}}

\newcommand{\M}{{\mathcal M}}
\newcommand{\T}{{\mathcal T}}

\newcommand{\C}{{\mathbb C}}
\newcommand{\R}{{\mathbb R}}

\title{The complex geometry of Teichm\"uller spaces and bounded
  symmetric domains}

\author{Stergios Antonakoudis *} 

\thanks{* Presented to the 6th Ahlfors-Bers Colloquium at Yale, New
  Haven CT, 23-26 October 2014.}

\begin{document}

  \maketitle


\section{Introduction}\label{sec:intro}
\noindent We study isometric maps between Teichm\"uller spaces
$\T_{g,n} \subset \C^{3g-3+n}$ and bounded symmetric domains
$\mathcal{B}\subset \C^N$ in their intrinsic Kobayashi metric. From a
complex analytic perspective, these two important classes of geometric
spaces have several features in common but also exhibit many
differences. The focus here is on recent results proved by the author;
we give a list of open questions at the end.

In a nutshell, we will see that Teichm\"uller spaces equipped with
their intrinsic Kobayashi metric exhibit a remarkable rigidity
property reminiscent of rank one bounded symmetric domains - in
particular, we will show that isometric disks are Teichm\"uller
disks. However, we will see that Teichm\"uller spaces and bounded
symmetric domains \textit{do not mix} isometrically so long as both
have dimension two or more.

The proofs of these results, although technically different, use the
common theme of \textit{complexification} and \textit{realification};
they also involve ideas from geometric topology.
\section{The setting}\label{sec:intro}
\begin{figure}[h]
  \centering
  \includegraphics[scale=0.3]{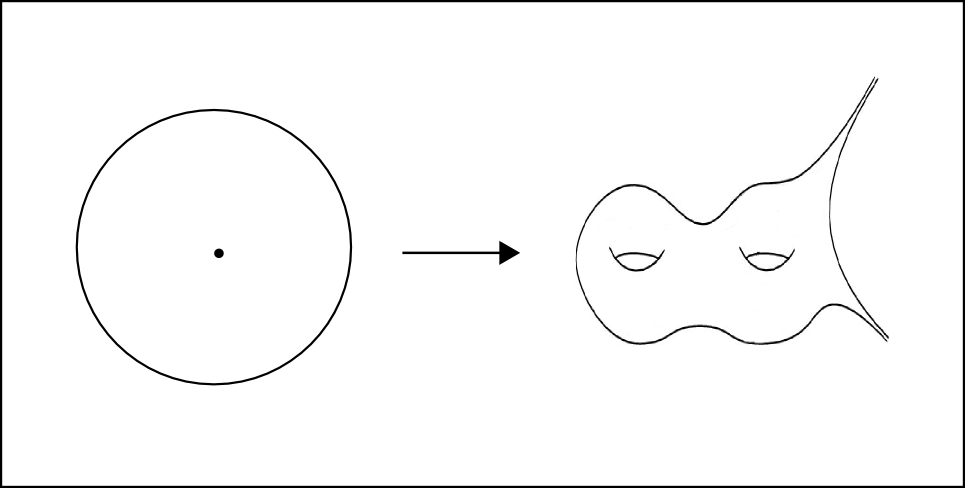}
  \caption{Universal covering $\pi:\Delta \rightarrow X = \Delta / \Gamma$}
  \label{covering}
\end{figure}

Let $X$ be a hyperbolic Riemann surface of finite type homeomorphic to
a fixed oriented topological surface $\Sigma_{g,n}$ of genus $g$ with
$n$ punctures. More concretely, we can present $X$ as a quotient space
$X = \Delta / \Gamma$, where $\Gamma \leq \text{Aut}(\Delta)$ is
discrete group of automorphisms of the unit disk
$\Delta\cong\{~z\in \C : |z| < 1~\}$ and
$\pi:\Delta \rightarrow X = \Delta / \Gamma$ is the universal covering
map.

The unit disk $\Delta$ is equipped with a metric $|dz|/(1-|z|^2)$ of
constant curvature, known as the Poincar\'e metric, which we shall
denote by $\mathbb{CH}^1$ and refer to as the complex hyperbolic
line. The group $\text{Aut}(\Delta)$ can be identified with the group
$\text{Isom}^{+}(\mathbb{CH}^1)$ of orientation preserving isometries
of $\mathbb{CH}^1$, hence we can endow $X = \Delta / \Gamma$ with a
finite-volume metric of constant curvature.~\\

The moduli space $\M_{g,n}$ parametrizing isomorphism classes of
Riemann surfaces $X$ has a similar description. It is a complex
quasi-projective variety which we can present as the quotient
$\M_{g,n} = \T_{g,n} / \text{Mod}_{g,n}$, where
$\text{Mod}_{g,n} \leq \text{Aut}(\T_{g,n})$ is a discrete group of
automorphisms of a contractible bounded domain
$\T_{g,n} \subset \C^{3g-3+n}$.

Teichm\"uller space $\T_{g,n}$ which parametrizes isomorphism classes
of marked Riemann surfaces is, therefore, the orbifold universal cover
of the moduli space of curves $\M_{g,n}$ and it is naturally a complex
manifold of dimension $3g-3+n$. It is equipped with a complete
intrinsic metric - the Teichm\"uller metric - which endows $\M_{g,n}$
with the structure of a finite-volume complex orbifold. It is known
that Teichm\"uller space can be realized as a bounded domain
$\T_{g,n} \subset\C^{3g-3+n}$ by the Bers embeddings.~\cite{Bers:ts:survey}~\\

Classically, another class of complex spaces admitting a similar
description is that of locally symmetric varieties $\mathcal{V}$ (of
non-compact type), which we can present as the quotient
$\mathcal{V} = \mathcal{B} / \Gamma$, where
$\Gamma\leq \text{Aut}(\mathcal{B})$ is a lattice, a discrete group of
automorphisms of a bounded symmetric domain $\mathcal{B}\subset \C^N$.

Let $\mathcal{B}\subset\C^{N}$ be a bounded domain; we call
$\mathcal{B}$ a bounded symmetric domain if every point
$p\in \mathcal{B}$ is an \textit{isolated} fixed point of a
holomorphic involution
$\sigma_{p} : \mathcal{B}\rightarrow \mathcal{B}$, with
$\sigma_{p}^2=\text{id}_{\mathcal{B}}$. Bounded symmetric domains are
contractible and homogeneous as complex manifolds. The simplest
example is given by the unit disk $\Delta \cong \mathbb{CH}^1$, which
is in fact the unique (up to isomorphism) contractible bounded domain
of complex dimension one. It is classically known that all Hermitian
symmetric spaces of non-compact type can be realized as bounded
symmetric domains $\mathcal{B}\subset\C^{N}$ by the Harish-Chandra
embeddings.~\cite{Helgason:book:dglgss}~\\

A feature that Teichm\"uller spaces and bounded symmetric domains have
in common is that they contain holomorphic isometric copies of
$\mathbb{CH}^1$ through every point and complex direction; in
particular, in complex dimension one, Teichm\"uller spaces and bounded
symmetric domains coincide.  However, in higher dimensions, the
situation is quite different. H. L. Royden proved that, when
$\text{dim}_{\C}\T_{g,n} \geq 2$, $\text{Aut}(\T_{g,n})$ is discrete
and therefore $\T_{g,n}$ is not a symmetric
space.~\cite{Royden:metric} Central to Royden's work was the use of
the intrinsic Kobayashi metric of $\T_{g,n}$.

\section{The Kobayashi metric}\label{sec:kobayashi-metric}

Let $\mathcal{B}\subset\C^{N}$ be a bounded domain, its intrinsic
Kobayashi metric is the \textit{largest} complex Finsler metric such
that every holomorphic map $f: \mathbb{CH}^1 \rightarrow \mathcal{B}$
is non-expanding: $||f'(0)||_{\mathcal{B}}\leq 1$. It determines both
a family of norms $||\cdot||_{\mathcal{B}}$ on the tangent bundle
$T\mathcal{B}$ and a distance $d_{\mathcal{B}}(\cdot,\cdot)$ on pairs
of points.~\cite{Kobayashi:book:hyperbolic}

We recall that Schwarz lemma shows that every holomorphic map
$f:\mathbb{CH}^1\rightarrow \mathbb{CH}^1$ is non-expanding. The
Kobayashi metric provides a natural generalisation - it has the
fundamental property that every holomorphic map between complex
domains is non-expanding and, in particular, every holomorphic
automorphism is an isometry. The Kobayashi metric of complex domain
depends only on its structure as a complex manifold.~\\

\begin{examples}~\\
  1. $\mathbb{CH}^1$ realises the unit disk $\Delta$ with its
  Kobayashi metric. The Kobayashi metric on the unit ball
  $\mathbb{CH}^2\cong\{~~(z,w)~~ | ~~ |z|^2 + |w|^2 < 1~~ \} \subset
  \C^2$ coincides with its unique (complete) invariant Ka\"ehler metric
  of constant holomorphic curvature -4.~\\ 
  2. The Kobayashi metric on the bi-disk $\mathbb{CH}^1\times\mathbb{CH}^1$ 
  coincides with the sup-metric of the two factors. It is a complex Finsler metric; 
  it is not a Hermitian metric.~\\
  3. The Kobayashi metric on $\T_{g,n}$ coincides with the classical
  Teichm\"uller metric, which endows $\T_{g,n}$ with the structure of
  a complete geodesic metric space.
\end{examples}
Incidentally, examples 1 and 2 above describe all bounded symmetric
domains up to isomorphism in complex dimensions one and two. We will
discuss example 3 in more detail below.

\section{Main results}\label{sec:results}

An important feature of the Kobayashi metric of Teichm\"uller space is
that every holomorphic map $f:\mathbb{CH}^1 \hookrightarrow \T_{g,n}$
such that $df$ is an isometry on tangent spaces is \textit{totally
  geodesic}: it sends real geodesics to real geodesics preserving
their length. Moreover, there are such holomorphic isometries, known
as Teichm\"uller disks, through every point in every complex
direction.

\subsection*{Holomorphic rigidity} Our first result is the
following:~\footnote{Theorem \ref{thm:disks:intro} solves problem 5.3
  from \cite{Fletcher:Markovic:survey}.}
\begin{thm}\label{thm:disks:intro}
  Every totally geodesic isometry $f: \mathbb{CH}^1 \hookrightarrow
  \T_{g,n}$ for the Kobayashi metric is either holomorphic or
  anti-holomorphic. In particular, it is a Teichm\"uller disk.
 \end{thm}

 This result is classically known for bounded symmetric domains with
 rank one and, more generally, for \textit{strictly} convex bounded
 domains. However, it is not true for bounded symmetric domains with
 rank two or more.  Our proof of Theorem~\ref{thm:disks:intro}
 recovers these classical results along with Teichm\"uller spaces by
 providing a more geometric approach.

 Theorem~\ref{thm:disks:intro} shows that the intrinsic
 Teichm\"uller-Kobayashi metric of $\T_{g,n}$ determines its natural
 structure as a complex manifold.~\\

 The following corollary follows easily from the theorem above.
\begin{cor}
  Every totally geodesic isometry $f: \T_{g,n} \hookrightarrow \T_{h,m}$ 
  is either holomorphic or anti-holomorphic.
\end{cor}
We note that, indeed, there are many holomorphic isometries
$f: \T_{g,n} \hookrightarrow \T_{h,m}$ between Teichm\"uller spaces
$\T_{g,n}$,$\T_{h,m}$ in their Kobayashi metric, induced by pulling
back complex structures from a fixed topological covering map
$\psi: \Sigma_{h,m} \rightarrow \Sigma_{g,n}$ of the underlying
topological surfaces $\Sigma_{g,n}$,$\Sigma_{h,m}$.~\cite{Kra:survey}

\subsection*{Symmetric spaces vs Teichm\"uller spaces}
Like Teichm\"uller spaces there are also many holomorphic isometries
$f: \mathcal{B} \hookrightarrow \widetilde{\mathcal{B}}$ between
bounded symmetric domains $\mathcal{B},\widetilde{\mathcal{B}}$ in
their Kobayashi metric.~\cite{Helgason:book:dglgss} However, in
dimension two or more, Teichm\"uller spaces and bounded symmetric
domains \textit{do not mix} isometrically. 

More precisely, we prove:
\begin{thm}\label{thm:bsds:intro}
  Let $\mathcal{B}$ be a bounded symmetric domain and $\T_{g,n}$ be a
  Teichm\"uller space with $\text{dim}_{\C}\mathcal{B},\text{dim}_{\C}\T_{g,n}
  \geq 2$. There are no holomorphic isometric immersions \[\mathcal{B}
  \xhookrightarrow{~~~f~~~} \T_{g,n} \quad \text{or} \quad \T_{g,n}
  \xhookrightarrow{~~~f~~~} \mathcal{B}\] such that $df$ is an isometry for the
  Kobayashi norms on tangent spaces.
\end{thm}

We record the following special case.
\begin{thm}\label{thm:balls:intro}
  There is no holomorphic isometry $f: \mathbb{CH}^2 \hookrightarrow
  \T_{g,n}$ for the Kobayashi metric.
\end{thm}

We also have a similar result for submersions:
\begin{thm}\label{thm:bsds:dual:intro}
  Let $\mathcal{B}$ and $\T_{g,n}$ be as in Theorem~\ref{thm:bsds:intro}. There
  are no holomorphic isometric submersions \[\mathcal{B} \xtwoheadrightarrow{g}
  \T_{g,n} \quad \text{or} \quad \T_{g,n}\xtwoheadrightarrow{g} \mathcal{B}\]
  such that $dg^{*}$ is an isometry for the \textit{dual} Kobayashi
  norms on cotangent spaces.~\\
\end{thm}

\subsection*{Remarks}~\\
\noindent 1. The existence of isometrically immersed curves, known as
Teichm\"uller curves, in $\M_{g,n}$ has far-reaching applications in
the dynamics of billiards in rational
polygons.~\cite{Veech:triangles},~\cite{McMullen:bild} The following
immediate Corollary of Theorem~\ref{thm:bsds:intro} shows that there
are no higher dimensional, locally symmetric, analogues of
Teichm\"uller curves.

\begin{cor}\label{cor-symmetric} 
  There is no locally symmetric variety $\mathcal{V}$ isometrically
  immersed in the moduli space of curves $\M_{g,n}$, nor is there an
  isometric copy of $\M_{g,n}$ in $\mathcal{V}$, for the Kobayashi
  metrics, so long as both have dimension two or more.
\end{cor}
~\\
\noindent 2. Torelli maps, associating to a marked Riemann surface the
Jacobians of its finite covers, give rise to holomorphic maps
$\T_{g,n} \xrightarrow{~~~\tau~~~} \mathcal{H}_h$ into bounded
symmetric domains (Siegel spaces). It is known that these maps are
isometric for the Kobayashi metric in some
directions\cite{Kra:abelian}, but strictly contracting in most
directions.~\cite{McMullen:covs}

~\\
\noindent 3. It is known that there are holomorphic isometric
submersions $\T_{g,n} \xtwoheadrightarrow{g} \mathbb{CH}^1$, which are
of the form $g=\rho \circ \tau$, where $\tau$ is the Torelli map
$\T_{g,n} \xrightarrow{~~~\tau~~~} \mathcal{H}_g$ to the Siegel
upper-half space and
$ \mathcal{H}_{g} \xtwoheadrightarrow{\rho} \mathbb{CH}^1$ is a
holomorphic isometric submersion.

~\\
For further details and proofs, we refer
to~\cite{Antonakoudis:disks},\cite{Antonakoudis:birational},\cite{Antonakoudis:symmetric}. In
this paper, we focus on explaining the proofs of
Theorem~\ref{thm:disks:intro} and Theorem~\ref{thm:balls:intro} using
the common theme of \textit{complexification} and
\textit{realification}. We start with some preliminaries on
Teichm\"uller spaces and their complex and real geodesics in the
intrinsic Teichm\"uller-Kobayashi metric.





\addtocontents{toc}{\protect\setcounter{tocdepth}{1}}

\section{Preliminaries in Teichm\"uller theory}\label{sec:prelim}
\subsection*{Teichm\"uller space}~\cite{Gardiner:Lakic:book},~\cite{Hubbard:book:T1}
Let $\Sigma_{g,n}$ be a connected, oriented surface of genus $g$ and
$n$ punctures and $\T_{g,n}$ denote the Teichm\"uller space of Riemann
surfaces marked by $\Sigma_{g,n}$. A point in $\T_{g,n}$ is specified
by an orientation preserving homeorphism
$\phi: \Sigma_{g,n} \rightarrow X$ to a Riemann surface of finite
type, up to a natural equivalence relation\footnote{Two marked Riemann
  surfaces $ \phi: \Sigma_{g,n} \rightarrow X$,
  $\psi: \Sigma_{g,n}\rightarrow Y$ are equivalent if
  $\psi \circ {\phi}^{-1}: X \rightarrow Y$ is isotopic to a
  holomorphic bijection.}.

Teichm\"uller space $\T_{g,n}$ is naturally a complex manifold of
dimension $3g-3+n$ and forgetting the marking realises $\T_{g,n}$ as
the complex orbifold universal cover of the moduli space $\M_{g,n}$.
When it is clear from the context we often denote a point specified by
$\phi: \Sigma_{g,n} \rightarrow X$ simply by $X$.

For each $X \in \T_{g,n}$, we let $Q(X)$ denote the space of
holomorphic quadratic differentials $q=q(z)(dz)^2$ on $X$ with finite
total mass: $ ||q||_{1} = \int_{X} |q(z)||dz|^2 < +\infty$, which
means that $q$ has at worse simple poles at the punctures of $X$.

The tangent and cotangent spaces to Teichm\"uller space at
$X\in \T_{g,n}$ are described in terms of the natural pairing
$ (q,\mu) \mapsto \int_{X} q\mu$ between the space $Q(X)$ and the
space $M(X)$ of $L^{\infty}$-measurable Beltrami differentials on $X$;
in particular, the tangent $T_{X} \T_{g,n}$ and cotangent
$T_{X}^{*} \T_{g,n}$ spaces are naturally isomorphic to
$M(X)/Q(X)^{\perp}$ and $Q(X)$, respectively.

The Teichm\"uller-Kobayashi metric on $\T_{g,n}$ is given by norm
duality on the tangent space $T_{X}\T_{g,n}$ from the norm
$||q||_{1} = \int_{X} |q|$ on the cotangent space $Q(X)$ at $X$. The
corresponding distance function is given by the formula
$d_{\T_{g,n}}(X,Y) = \inf \frac{1}{2} \log K(\phi)$ and measures the
minimal dilatation $K(\phi)$ of a quasiconformal map
$\phi: X \rightarrow Y$ respecting their markings.


\subsection*{Measured foliations}
Let $\mathcal{MF}_{g,n}$ denote the space of equivalent
classes\footnote{Two measured foliations $\mathcal{F},\mathcal{G}$ are
  equivalent $\mathcal{F}\thicksim \mathcal{G}$ if they differ by a
  finite sequence of Whitehead moves followed by an isotopy of
  $\Sigma_{g,n}$, preserving their transverse measures.} of nonzero
(singular) measured foliations on $\Sigma_{g,n}$. It is known that
$\mathcal{MF}_{g,n}$ has the structure of a \textit{piecewise linear}
manifold, which is homeomorphic to
$\mathbb{R}^{6g-6+2n}\setminus{\{0\}}$.~\cite{FLP}

The geometric intersection number of a pair of measured foliations
$\mathcal{F},\mathcal{G}$, denoted by $i(\mathcal{F},\mathcal{G})$,
induces a continuous map
$i(\cdot,\cdot): \mathcal{MF}_{g,n} \times \mathcal{MF}_{g,n}
\rightarrow \mathbb{R}_{\geq 0}$,
which extends the geometric intersection pairing on the space of
(isotopy classes of) simple closed curves on
$\Sigma_{g,n}$.~\cite{Bonahon:currents}

Given $\mathcal{F} \in \mathcal{MF}_{g,n}$ and $X \in \T_{g,n}$, we
let $\lambda(\mathcal{F},X)$ denote the \textit{extremal length} of
$\mathcal{F}$ on the Riemann surface $X$ given by the formula
$\lambda(\mathcal{F},X)= \sup
\frac{\ell_{\rho}(\mathcal{F})^2}{\text{area}(\rho)}$,
where $\ell_{\rho} (\mathcal{F})$ denotes the $\rho$-length of
$\mathcal{F}$ and the supremum is over all (Borel-measurable)
conformal metrics $\rho$ of finite area on $X$.

Each nonzero quadratic differential $q \in Q(X)$ induces a conformal
metric $|q|$ on $X$, which is non-singular of zero curvature away from 
the zeros of $q$, and a measured foliation $\mathcal{F}(q)$ tangent to
vectors $v=v(z)\frac{\partial}{\partial z}$ with $q(v)=q(z)(v(z))^2 <0$. 
The transverse measure of the foliation $\mathcal{F}(q)$ is (locally)
given by integrating $|\text{Re}(\sqrt{q})|$ along arcs transverse to
its leaves.

We refer to $\mathcal{F}(q)$ as the vertical measured foliation
induced from $(X,q)$. In local coordinates, where $q=dz^2$ (such
coordinates exist away from the zeros of $q$), the metric $|q|$
coincides with the Euclidean metric $|dz|$ in the plane and the
measured foliation $\mathcal{F}(q)$ has leaves given by vertical lines
and transverse measure by the total horizontal variation
$|\text{Re}(dz)|$. We note that the measured foliation
$\mathcal{F}(-q)$ has (horizontal) leaves orthogonal to
$\mathcal{F}(q)$ and the product of their transverse measures is just
the area form of the conformal metric $|q|$ induced from $q$. 

When it is clear from the context we often identify the measured
foliation $\mathcal{F}(q)$ with its equivalence class in
$\mathcal{MF}_{g,n}$. The following fundamental theorem relates
quadratic differentials and measured foliations on fixed Riemann
surface.

\begin{thm}(\cite{Hubbard:Masur:fol};Hubbard-Masur)\label{thm:hubbard:masur}
  Let $X \in \T_{g,n}$; the map $q \mapsto \mathcal{F}(q)$ induces a
  homeomorphism $Q(X) \setminus \{0\} \cong \mathcal{MF}_{g,n}$.
  Moreover, $|q|$ is the unique extremal metric for $\mathcal{F}(q)$
  on $X$ and its extremal length is given by the formula
  $\lambda(\mathcal{F},X) = ||q||_1$.
\end{thm}

\subsection*{Complex geodesics}\label{sec:complex-geodesics}
We denote by $Q\T_{g,n} \cong T^{*}\T_{g,n}$ the complex vector-bundle
of holomorphic quadratic differentials over $\T_{g,n}$ and by
$Q_1\T_{g,n}$ the associated sphere-bundle of quadratic differentials
with unit mass. There is a natural norm-preserving action of
$\text{SL}_2(\mathbb{R})$ on $Q\T_{g,n}$, with the diagonal matrices
giving the (co-)geodesic flow.

For each $(X,q)\in Q_1\T_{g,n}$, the orbit
$\text{SL}_2(\mathbb{R}) \cdot (X,q) \subset Q_1\T_{g,n}$ induces a
holomorphic totally geodesic isometry
\[\mathbb{CH}^1 \cong
  \text{SO}_2(\mathbb{R})\setminus\text{SL}_2(\mathbb{R})
  \hookrightarrow \T_{g,n}\]
which we refer to as the \textit{Teichm\"uller disk} generated by
$(X,q)$.

\subsection*{Real geodesics}\label{sec:real-geodesics}
Let $\gamma: [0,\infty) \rightarrow \T_{g,n}$ be a Teichm\"uller
geodesic ray with unit speed, which has a unique lift
$\widetilde{\gamma}(t)=(X_t,q_t) \in Q_1\T_{g,n}$ such that
$\gamma(t)=X_t$ and
$\widetilde{\gamma}(t) = \text{diag}(e^{t}, e^{-t})\cdot (X_0,q_0)$
for $t\in \R_{\geq 0}$.

The map $q \mapsto (\mathcal{F}(q),\mathcal{F}(-q))$ gives an
embedding
\[Q \T_{g,n} \hookrightarrow \mathcal{MF}_{g,n} \times
  \mathcal{MF}_{g,n}\]
which satisfies $||q||_1=i(\mathcal{F}(q),\mathcal{F}(-q))$ and sends
the lift $\widetilde{\gamma}(t)=(X_t,q_t)$ of the Teichm\"uller
geodesic ray $\gamma$ to a path of the form
$(e^t\mathcal{F}(q),e^{-t}\mathcal{F}(-q))$.

Let $X \in \T_{g,n}$ and let $q\in Q(X)$ generate a real Teichm\"uller
geodesic $\gamma$ with $\gamma(0)=X$. The geodesic ray $\gamma$
extends uniquely to a holomorphic totally geodesic isometry
$\gamma_{\C}: \Delta \cong \mathbb{CH}^1 \hookrightarrow \T_{g,n}$
satisfying $\gamma(t)=\gamma_{\C}(\tanh(t))$ for $t \in \R$; the
Teichm\"uller geodesic generated by the quadratic differential
$e^{i\theta}q \in Q(X)$, with $\theta \in\mathbb{R}/2\pi\mathbb{Z}$,
is given by the map
$t \mapsto \gamma_{\C} (e^{-i\theta}\text{tanh}(t))$, $t \in \R$.





\section{Holomorphic rigidity}\label{sec:disks}                                                                              

In this section we prove:
\begin{thm}\label{thm:disks}
  Every totally geodesic isometry
  $f: \mathbb{CH}^1 \hookrightarrow \T_{g,n}$ for the Kobayashi metric
  is either holomorphic or anti-holomorphic. In particular, it is a
  Teichm\"uller disk.
\end{thm}
The proof of the theorem uses the idea of \textit{complexification}
and leverages the following two facts.  Firstly, a complete real
geodesic in $\T_{g,n}$ is contained in a unique holomorphic
Teichm\"uller disk; and secondly, a holomorphic family
$\{f_t\}_{t\in\Delta}$ of \textit{essentially proper} holomorphic maps
$f_t : \mathbb{CH}^1 \rightarrow \T_{g,n}$ is \textit{trivial}:
$f_t = f_0$ for $t\in\Delta$ (Sullivan's rigidity theorem, see
~\cite{Tanigawa:holomap} for a precise statement and proof).

\subsection*{Outline of the proof} 
Let $\gamma \subset \mathbb{CH}^1$ be a complete real geodesic and
denote by
$\gamma_{\C} \subset \mathbb{CH}^1\times\overline{\mathbb{CH}^1}$ its
\textit{maximal} holomorphic extension to the bi-disk. We note that
$\gamma_{\C} \cong \mathbb{CH}^1$ and we define $F|_{\gamma_{\C}}$ to
be the \textit{unique} holomorphic extension of $f|_{\gamma}$, which
is a Teichm\"uller disk.

Applying this construction to all (real) geodesics in $\mathbb{CH}^1$,
we will deduce that $f: \mathbb{CH}^1 \rightarrow \T_{g,n}$ extends to
a \textit{holomorphic} map
$F:\mathbb{CH}^1\times\overline{\mathbb{CH}^1}\rightarrow \T_{g,n}$
such that $f(z)=F(z,z)$ for $z\in \Delta \cong \mathbb{CH}^1$. Using
that $f$ is totally geodesic, we will show that $F$ is
\textit{essentially} proper and hence, by Sullivan's rigidity theorem,
we will conclude that either $F(z,w) = F(z,z)$ or $F(z,w)=F(w,w)$, for
all $(z,w)\in \mathbb{CH}^1\times\overline{\mathbb{CH}^1}$.\qed
~\\

\large
\[
  \xymatrix{
    \mathbb{CH}^1 \times \overline{\mathbb{CH}^1} \ar@{->}^F[rd] \\
    \mathbb{CH}^1 \ar@{^{(}->}^{\delta}[u] \ar@{^{(}->}^f[r] &\T_{g,n} }
\]
~\\
\normalsize

We start with some preliminary constructions.
\subsection*{The totally real diagonal}
Let $\overline{\mathbb{CH}^1}$ be the complex hyperbolic line with its
conjugate complex structure. The identity map is a \textit{canonical}
anti-holomorphic isomorphism
$\mathbb{CH}^{1}\cong\overline{\mathbb{CH}^1}$ and its graph is a
totally real embedding
$\delta: \mathbb{CH}^{1} \hookrightarrow
\mathbb{CH}^{1}\times\overline{\mathbb{CH}^1}$,
given by $\delta(z)=(z,z)$ for $z\in \Delta\cong \mathbb{CH}^{1}$. We
call $\delta(\mathbb{CH}^{1})$ the \textit{totally real diagonal}.

\subsection*{Geodesics and graphs of reflections}
Let $\mathcal{G}$ denote the set of all real, unoriented, complete
geodesics $\gamma \subset \mathbb{CH}^{1}$. In order to describe their
\textit{maximal} holomorphic extensions
$\gamma_{\C} \subset \mathbb{CH}^1\times\overline{\mathbb{CH}^1}$,
such that $\gamma_{\C} \cap \delta(\mathbb{CH}^1) = \delta(\gamma)$,
it is convenient to parametrize $\mathcal{G}$ in terms of the set
$\mathcal{R}$ of hyperbolic reflections of $\mathbb{CH}^1$ - or
equivalently, the set of anti-holomorphic involutions of
$\mathbb{CH}^1$. The map that associates a reflection
$r\in \mathcal{R}$ with the set
$\gamma = \text{Fix}(r) \subset \mathbb{CH}^{1}$ of its fixed points
gives a bijection between $\mathcal{R}$ and $\mathcal{G}$.

Let $r\in\mathcal{R}$ and denote its graph by
$\Gamma_r\subset\mathbb{CH}^{1}\times\overline{\mathbb{CH}^{1}}$;
there is a natural holomorphic isomorphism
$\mathbb{CH}^1 \cong \Gamma_r$, given by $z \mapsto (z,r(z))$ for
$z\in \Delta \cong \mathbb{CH}^1$. We note that $\Gamma_r$ is the
\textit{maximal} holomorphic extension $\gamma_{\C}$ of the geodesic
$\gamma = \text{Fix}(r)$ to the bi-disk and it is \textit{uniquely}
determined by the property
$\gamma_{\C} \cap \delta(\mathbb{CH}^1) = \delta(\gamma)$.

\subsection*{The foliation by graphs of reflections}
The union of the graphs of reflections
$\bigcup_{r\in\mathcal{R}}\Gamma_r$ gives rise to a (singular)
foliation of $\mathbb{CH}^{1}\times\overline{\mathbb{CH}^{1}}$ with
holomorphic leaves $\Gamma_r$ parametrized by the set
$\mathcal{R}$. We have
$\displaystyle\Gamma_r \cap \delta(\mathbb{CH}^{1}) =
\delta(\text{Fix}(r))$ for all $r\in\mathcal{R}$, and
\begin{equation}\label{eq:leaves}
\displaystyle\Gamma_r \cap \Gamma_s = \delta(\text{Fix}(r) \cap \text{Fix}(s))
\end{equation}
which is either empty or a single point for all $r,s \in \mathcal{R}$
with $r \neq s$. In particular, the foliation is smooth in the
complement of the totally real diagonal $\delta(\mathbb{CH}^{1})$.

We emphasize that the following simple observation plays a key role in
the proof of the theorem. For all $r\in\mathcal{R}$:
\begin{equation}
  \label{eq:flip}
  (z,w) \in \Gamma_r \iff (w,z) \in \Gamma_r 
\end{equation}
\subsection*{Geodesics and the Klein model}\label{sec:klein-model}
The Klein model gives a real-analytic identification
$\mathbb{CH}^1\cong\mathbb{RH}^2\subset\R^2$ with an open disk in
$\R^2$. It has the nice property that the hyperbolic geodesics are
affine straight lines intersecting the disk.~\cite{Ratcliffe:book}
\begin{remark}
  The holomorphic foliation by graphs of reflections defines a
  \textit{canonical} complex structure in a neighborhood of the zero
  section of the tangent bundle of $\mathbb{RH}^2$.
\end{remark}
The description of geodesics in the Klein model is convenient in the
light of the following theorem of S.~Bernstein.

\begin{thm}\label{thm:bernstein}(\cite{Ahiezer:Ronkin:Bernstein};~S.~Bernstein)
  Let $M$ be a complex manifold, $f: [0,1]^2 \rightarrow M$ a map from
  the square $[0,1]^2 \subset \R^2$ into $M$ and $E\subset \C$ an
  ellipse with foci at $0,1$. If there are holomorphic maps
  $F_{\ell} : E \rightarrow M$ such that
  $F_{\ell}|_{[0,1]} = f|_{\ell}$, for all vertical and horizontal
  slices $\ell\cong [0,1]$ of $[0,1]^2$, then $f$ has a unique
  holomorphic extension in a neighborhood of $[0,1]^2$ in $\C^2$.
\end{thm}
We use this to prove:
\begin{lem}\label{lem-real-analytic}
  Every totally geodesic isometry
  $f:\mathbb{CH}^{1} \hookrightarrow \T_{g,n}$ admits a unique
  holomorphic extension in a neighborhood of the totally real diagonal
  $\delta(\mathbb{CH}^{1})\subset\mathbb{CH}^{1}\times\overline{\mathbb{CH}^{1}}$.
\end{lem}
\begin{proof}[Proof of~\ref{lem-real-analytic}]
  Using the fact that analyticity is a local property and the
  description of geodesics in the Klein model of $\mathbb{RH}^2$, we
  can assume - without loss of generality - that the map $f$ is
  defined in a neighborhood of the unit square $[0,1]^2$ in $\R^2$ and
  has the property that its restriction on every horizontal and
  vertical line segment $\ell \cong [0,1]$ is a real-analytic
  parametrization of a Teichm\"uller geodesic segment. Moreover, we
  can also assume that the lengths of all these segments, measured in
  the Teichm\"uller metric, are uniformly bounded from above and from
  below away from zero.

  Since every segment of a Teichm\"uller geodesic extends to a
  (holomorphic) Teichm\"uller disk in $\T_{g,n}$, there exists an
  ellipse $E\subset\C$ with foci at $0$,$1$ such that the restrictions
  $f|_{\ell}$ extend to holomorphic maps
  $F_{\ell}: E \rightarrow \T_{g,n}$ for all horizontal and vertical
  line segments $\ell\cong [0,1]$ of $[0,1]^2$.  Hence, the proof of
  the lemma follows from Theorem~\ref{thm:bernstein}.
\end{proof}

\subsection*{Proof of Theorem~\ref{thm:disks}}~\\
Let $f:\mathbb{CH}^{1} \hookrightarrow \T_{g,n}$ be a totally geodesic
isometry. Applying Lemma~\ref{lem-real-analytic}, we deduce that $f$
has a \textit{unique} holomorphic extension in a neighborhood of the
totally real diagonal
$\delta(\mathbb{CH}^{1})\subset\mathbb{CH}^{1}\times\overline{\mathbb{CH}^{1}}$. We
will show that $f$ extends to a
holomorphic map from $\mathbb{CH}^{1}\times\overline{\mathbb{CH}^{1}}$ to $\T_{g,n}$.~\\
We start by defining a \textit{new} map
$F:\mathbb{CH}^{1}\times\overline{\mathbb{CH}^{1}} \rightarrow
\T_{g,n}$, satisfying:

1. $F(z,z)=f(z)$ for all $z\in\Delta \cong \mathbb{CH}^{1}$.

2. $F|_{\Gamma_r}$ is the \textit{unique} holomorphic extension of
$f|_{\text{Fix}(r)}$ for all $r\in\mathcal{R}$.
~\\
Let $r\in\mathcal{R}$ be a reflection. There is a \textit{unique}
(holomorphic) Teichm\"uller disk
$\phi_r:\mathbb{CH}^{1}\hookrightarrow\T_{g,n}$ such that the
intersection
$\phi_r(\mathbb{CH}^1)\cap f(\mathbb{CH}^{1})\subset \T_{g,n}$
contains the Teichm\"uller geodesic $f(\text{Fix}(r))$ and
$\phi_r(z)=f(z)$ for all $z\in\text{Fix}(r)$.

We define $F$ by $F(z,r(z))=\phi_r(z)$ for $z\in\mathbb{CH}^{1}$ and
$r\in\mathcal{R}$; equation~(\ref{eq:leaves}) shows that $F$ is
well-defined and satisfies conditions (1) and (2) above.

We \textit{claim} that
$F:\mathbb{CH}^{1}\times\overline{\mathbb{CH}^{1}}\rightarrow
\T_{g,n}$
is the \textit{unique} holomorphic extension of
$f: \mathbb{CH}^1 \hookrightarrow \T_{g,n}$ such that $F(z,z)=f(z)$
for $z \in\mathbb{CH}^1$.

\textit{Proof of claim}. We note that the restriction of $F$ on the
totally real diagonal $\delta(\mathbb{CH}^{1})$ agrees with $f$ and
that there is a \textit{unique} germ of holomorphic maps near
$\delta(\mathbb{CH}^{1})$ whose restriction on
$\delta(\mathbb{CH}^{1})$ coincides with $f$. Let us fix an element of
this germ $\tilde{F}$ defined on a neighborhood
$U\subset\mathbb{CH}^{1}\times\overline{\mathbb{CH}^{1}}$ of
$\delta(\mathbb{CH}^{1})$. For every $r\in\mathcal{R}$, the
restrictions of $F$ and $\tilde{F}$ on the intersection
$U_r= U \cap \Gamma_r$ are holomorphic and equal along the
real-analytic arc $U_r \cap \delta(\mathbb{CH}^{1}) \subset U_r$;
hence they are equal on $U_r$. Since
$\mathbb{CH}^{1}\times\overline{\mathbb{CH}^{1}} =
\bigcup_{r\in\mathcal{R}}\Gamma_r$,
we conclude that $F|_{U}=\tilde{F}$ and, in particular, $F$ is
holomorphic near the totally real diagonal $\delta(\mathbb{CH}^{1})$.
Since, in addition to that, $F$ is holomorphic along all the leaves
$\Gamma_r$ of the foliation, we deduce~\footnote{For a simple proof of
  this claim using the power series expansion of $F$ at
  $(0,0)\in\mathbb{CH}^{1}\times\overline{\mathbb{CH}^{1}}$,
  see~\cite[Lemma~2.2.11]{Hormander:book}.} that it is holomorphic at
all points of $\mathbb{CH}^{1}\times\overline{\mathbb{CH}^{1}}$.\qed

In order to finish the proof of the theorem, we use the \textit{key}
observation~(\ref{eq:flip}); which we recall as follows: the points
$(z,w)$ and $(w,z)$ are always contained in the same leaf $\Gamma_r$
of the foliation for all $z,w\in\Delta\cong\mathbb{CH}^{1}$. Using the
fact that the restriction of $F$ on every leaf $\Gamma_{r}$ is a
Teichm\"uller disk, we conclude that
$d_{\T_{g,n}}(F(z,w),F(w,z))=d_{\mathbb{CH}^{1}}(z,w)$.
  
Let $\theta \in \mathbb{R}/2\pi\mathbb{Z}$, it follows that at least
one of $\displaystyle F(\rho e^{i\theta},0)$ and
$\displaystyle F(0,\rho e^{i\theta})$ diverges in Teichm\"uller space
as $\rho \rightarrow 1$. In particular, there is a subset
$I\subset \mathbb{R}/2\pi\mathbb{Z}$ with positive measure such that
either $F(\rho e^{i\theta},0)$ or
$\displaystyle F(0,\rho e^{i\theta})$ diverges as $\rho \rightarrow 1$
for all $\theta \in I$.

We assume first that the former of the two is true. Using that
$F: \mathbb{CH}^1\times\overline{\mathbb{CH}^1} \rightarrow \T_{g,n}$
is holomorphic, we deduce from~\cite{Tanigawa:holomap} (Sullivan's
rigidity theorem) that the family $\{F(z,\overline{w})\}_{w\in\Delta}$
of holomorphic maps
$ F(\cdot,\overline{w}): \Delta\cong\mathbb{CH}^1 \rightarrow
\T_{g,n}$
for $w\in \Delta\cong\mathbb{CH}^1$ is \textit{trivial}. Therefore,
$F(z,0)=F(z,z)=f(z)$ for all $z\in \Delta$ and, in particular, $f$ is
holomorphic. If we assume that the latter of the two is true we
similarly deduce that $F(0,z)=F(z,z)=f(z)$ for all $z\in \Delta$ and,
in particular, $f$ is anti-holomorphic.\qed





\section{Extremal length geometry}\label{sec:balls}

In this section we prove:

\begin{thm}\label{thm:balls}
  There is no holomorphic isometry $f: \mathbb{CH}^2 \hookrightarrow
  \T_{g,n}$ for the Kobayashi metric.
\end{thm}

The proof of the theorem uses the idea of \textit{realification} and
leverages the fact that extremal length provides a link between the
geometry of Teichm\"uller geodesics and the geometric intersection
pairing for measured foliations.

\subsection*{Outline of the proof}
Using a theorem of
Slodkowski~\cite{Slodkowski:motions},~\cite{Earle:Kra:Krushkal}, we
deduce that such an isometry would be totally-geodesic - it would send
real geodesics in $\mathbb{CH}^2$ to Teichm\"uller geodesics in
$\T_{g,n}$ preserving their length. We can parametrize the set of
Teichm\"uller geodesic rays from any base point $X\in\T_{g,n}$, using
Theorem~\ref{thm:hubbard:masur}, by the subspace of measured
foliations $\mathcal{F}\in\mathcal{MF}_{g,n}$ with extremal length
$\lambda(\mathcal{F},X)=1$.

Assuming the existence of $f$, we consider pairs of measured
foliations that parametrize orthogonal geodesic rays in the image of a
\textit{totally real} geodesic hyperbolic plane
$\mathbb{RH}^2\subset\mathbb{CH}^2$.  We obtain a contradiction by
computing their geometric intersection number in two different ways.
~\\
\large
\[
\xymatrix{
\mathbb{CH}^2 \ar@{^{(}->}^f[r]
&\T_{g,n}\\
\mathbb{RH}^2  \ar@{^{(}->}[u] \ar@{^{(}->}[ur]}
\]
~\\
\normalsize

On the one hand, we use the geometry of complex hyperbolic horocycles
and extremal length to show that the geometric intersection number
does not depend on the choice of the totally real geodesic plane. On
the other hand, by a direct geometric argument we show that this is
impossible. More precisely, we have:
\begin{prop}\label{prop:intersection}
  Let $q\in Q_1\T_{g,n}$ and $\mathcal{G}\in \mathcal{MF}_{g,n}$.
  There exist $v_1, \ldots, v_N \in \C^{*}$ such that
  $i(\mathcal{F}(e^{i\theta}q),\mathcal{G})=\sum_{i=1}^{N} |
  \text{Re}(e^{i\theta/2}v_i)|$ for all $\theta \in
  \mathbb{R}/2\pi\mathbb{Z}$.
\end{prop}
The proof of the proposition is given at the end of the section.\qed
~\\

We start with preliminaries on compex hyperbolic and extremal length
horocycles.

\subsection*{Complex hyperbolic horocycles}
Let $\gamma: [0,\infty) \rightarrow \mathbb{CH}^2$ be a geodesic ray
with unit speed. Since $\mathbb{CH}^2$ is a homogeneous space, we have
$\gamma=\alpha \circ \gamma_1$, where $\gamma_1(t) = (\tanh(t),0)$,
for $t\geq 0$, and $\alpha$ is a holomorphic isometry of
$\mathbb{CH}^2$. Each geodesic ray is contained in the image of unique
holomorphic totally-geodesic isometry
$\gamma: \mathbb{CH}^1 \hookrightarrow \mathbb{CH}^2$ satisfying
$\gamma(t)=\phi(\tanh(t))$; in particular, $\phi_1 (z) = (z,0)$, for
$z\in \Delta\cong \mathbb{CH}^1$. We note that every complex geodesic
$\phi: \mathbb{CH}^1 \hookrightarrow \mathbb{CH}^2$ arises uniquely
(up to pre-composition with an automorphism of $\mathbb{CH}^1$) as the
intersection of the unit ball in $\mathbb{C}^2$ with a complex affine
line.

Associated to each geodesic ray
$\gamma: [0,\infty) \rightarrow \mathbb{CH}^2$ is a pair of transverse
foliations of $\mathbb{CH}^2$, one by real geodesics asymptotic to
$\gamma$ and another by complex hyperbolic horocycles asymptotic to
$\gamma$. For each $p\in \mathbb{CH}^2$ there exists a \textit{unique}
geodesic $\gamma_p: \R \rightarrow \mathbb{CH}^2$ and a
\textit{unique} time $t_p\in \R$ such that $\gamma_p(t_p)=p$ and
$\displaystyle \lim_{t\rightarrow
  \infty}d_{\mathbb{CH}^2}(\gamma(t),\gamma_p(t))\rightarrow 0$.
For each $s\in \R_{+}$, we define the set
$H(\gamma,s)= \{~~ p \in \mathbb{CH}^2 ~~|~~ \exp(t_p)=s ~~\}$. The
collection of subsets $\{H(\gamma,s)\}_{s\in \mathbb{R}_{+}}$ defines
the foliation of $\mathbb{CH}^2$ by \textit{complex hyperbolic
horocycles} asymptotic to $\gamma$.

\subsection*{Extremal length horocycles}
Let $\gamma: [0,\infty) \rightarrow \T_{g,n}$ be a Teichm\"uller
geodesic ray with unit speed. It has a unique lift to
$\widetilde{\gamma}(t)=(X_t,q_t) \in Q_1\T_{g,n}$, such that
$\gamma(t)=X_t$ and $\widetilde{\gamma}(t) = \text{diag}(e^{t}, e^{-t})\cdot (X_0,q_0)$.
The map $q \mapsto (\mathcal{F}(q),\mathcal{F}(-q))$ gives an
embedding
$Q\T_{g,n} \hookrightarrow \mathcal{MF}_{g,n} \times
\mathcal{MF}_{g,n}$
which satisfies $||q||_1=i(\mathcal{F}(q),\mathcal{F}(-q))$ and sends
the lift $\widetilde{\gamma}(t)=(X_t,q_t)$ of Teichm\"uller geodesic
ray $\gamma$ to a path of the form
$(e^t\mathcal{F}(q),e^{-t}\mathcal{F}(-q))$.

The later description of a Teichm\"uller geodesic and
Theorem~\ref{thm:hubbard:masur} show that the extremal length of
$\mathcal{F}(q_t)$ along $\gamma$ satisfies
$\lambda(\mathcal{F}(q_t),X_s)=e^{2(t-s)}$ for all $t,s\in \R_{+}$,
which motivates the following definition. For each
$\mathcal{F} \in \mathcal{MF}_{g,n}$ the \textit{extremal length
  horocycles} asymptotic to $\mathcal{F}$ are the level-sets of
extremal length
$H(\mathcal{F},s) = \{~~ X \in \T_{g,n} ~~|~~ \lambda(\mathcal{F},X)=s
~~\}$
for $s\in \R_{+}$. The collection of subsets
$\{H(\mathcal{F},s)\}_{s\in \mathbb{R}_{+}}$ defines the foliation of
$\T_{g,n}$ by \textit{extremal length horocycles} asymptotic to
$\mathcal{F}$.

There is transverse foliation of $\T_{g,n}$ by real Teichm\"uller
geodesics with lifts $(X_t, q_t)$ that satisfy
$\mathcal{F}(q_{t}) \in \R_{+}\cdot \mathcal{F}$. One might expect
that this foliation of $\T_{g,n}$ is analogous to the foliation of
$\mathbb{CH}^2$ by geodesics that are positively asymptotic to
$\gamma$. Although this is not always true, it is true for
\textit{generic} measured foliations
$\mathcal{F} \in \mathcal{MF}_{g,n}$.
\begin{thm}\label{thm:masur:ue}(\cite{Masur:ergodic:geodesics};~H.~Masur)
  Let $(X_t,q_t)$ and $(Y_t,p_t)$ be two Teichm\"uller geodesics and
  $\mathcal{F}(q_0)\in\mathcal{MF}_{g,n}$ be uniquely
  ergodic.~\footnote{A measured foliation $\mathcal{F}$ is
    \textit{uniquely ergodic} if it is minimal and admits a unique, up
    to scaling, transverse measure; in particular,
    $i(\gamma,\mathcal{F}) > 0 $ for all simple closed curves
    $\gamma$. Compare with ~\cite{Masur:ergodic:geodesics}.} Then
  $lim_{t\rightarrow\infty}d_{\T_{g,n}}(X_t,Y_t)\rightarrow 0$ if and
  only if $\mathcal{F}(q_0)=\mathcal{F}(p_0)$ in $\mathcal{MF}_{g,n}$
  and $\lambda(\mathcal{F}(q_0),X_0)=\lambda(\mathcal{F}(p_0),Y_0)$.
\end{thm}
\begin{remark}
  It is known that this result is not true for measured foliations
  that are not uniquely ergodic.
\end{remark}
\subsection*{Proof of Theorem~\ref{thm:balls}}
Let $f: \mathbb{CH}^2 \hookrightarrow \T_{g,n}$ be a holomorphic
isometry for the Kobayashi metric. We summarize the proof in the
following three steps:~\\

\noindent\textbf{1.} \textit{Asymptotic behavior of geodesics
  determines the extremal length horocycles.}

\noindent\textbf{2.} \textit{The geometry of horocycles determines the
  geometric intersection pairing.}

\noindent\textbf{3.} \textit{Get a contradiction by a direct
  computation of the geometric intersection pairing.}~\\

\noindent\textbf{Step 1.} Let $X =f((0,0)) \in \T_{g,n}$ and $q,p \in Q_1(X)$
unit area quadratic differentials generating the two Teichm\"uller
geodesic rays $f(\gamma_1)$,$f(\gamma_2)$, where $\gamma_1$,$\gamma_2$
are two orthogonal geodesic rays in $\mathbb{CH}^2$ contained in the
image of the totally real geodesic hyperbolic plane
$\mathbb{RH}^2\subset\mathbb{CH}^2$; explicitly, they are given by the
formulas $\gamma_1(t) = (\tanh(t),0)$, $\gamma_2(t)=(0,\tanh(t))$, for
$t\geq0$.

For every $(X,q) \in Q_1\T_{g,n}$ there is a dense set of
$\theta \in \mathbb{R}/2\pi\mathbb{Z}$ such that the measured
foliation $\mathcal{F}(e^{i\theta}q)$ is uniquely
ergodic~\cite{Chaika:Cheung:Masur:winning}; hence, we can assume
without loss of generality (up to a holomorphic automorphism of
$\mathbb{CH}^2$) that both $\mathcal{F}(q)$ and $\mathcal{F}(p)$ are
(minimal) uniquely ergodic measured foliations. In particular, we can
apply Theorem~\ref{thm:masur:ue} to study the extremal length
horocycles asymptotic to $\mathcal{F}(q)$ and $\mathcal{F}(p)$
respectively.

The complex hyperbolic horocycle $H(\gamma_1,1)$ is characterized by
the property that for the points $P \in H(\gamma_1,1)$ the geodesic
distance between $\gamma_P(t)$ and $\gamma_1(t)$ tends to zero as
$t \rightarrow +\infty$, where $\gamma_P(t)$ is the unique geodesic
with unit speed through $P$ that is positively asymptotic to
$\gamma_1$. Applying Theorem~\ref{thm:masur:ue} we conclude that:
\begin{equation}\label{eq:1}
f(\mathbb{CH}^2) \cap H(\mathcal{F}(q),1) = f( H(\gamma_1,1))
\end{equation}
\begin{equation}\label{eq:2}
f(\mathbb{CH}^2) \cap H(\mathcal{F}(p),1) = f( H(\gamma_2,1))  
\end{equation}

\noindent\textbf{Step 2.}  Let $\delta$ be the (unique) complete real 
geodesic in $\mathbb{CH}^2$, which is asymptotic to $\gamma_1$ in the
positive direction and to $\gamma_2$ in the negative direction,
i.e. its two endpoints are $(1,0),(0,1) \in \C^2$ in the boundary of
the unit ball. Let $P_1$ and $P_2$ be the two points where $\delta$
intersects the horocycles $H(\gamma_1,1)$ and $H(\gamma_2,1)$,
respectively. See~\ref{fig1}.

The image of $\delta$ under the map $f$ is a Teichm\"uller geodesic
which is parametrized by a pair of measured foliations
$\mathcal{F},\mathcal{G} \in \mathcal{MF}_{g,n}$ with
$i(\mathcal{F},\mathcal{G})=1$ and its unique lift to $Q_1\T_{g,n}$ is
given by $(e^t \mathcal{F},e^{-t}\mathcal{G})$, for $t\in \R$. Let
$\widetilde{P_i} = (e^{t_i} \mathcal{F},e^{-t_i}\mathcal{G})$, for
$i=1,2$, denote the lifts of $P_1,P_2$ along the geodesic
$\delta$. Then, the distance between the two points is given by
$d_{\mathbb{CH}^2}(P_1,P_2) = t_2-t_1$. From Step 1, we conclude that
$e^{t_1}\mathcal{F} = \mathcal{F}(q)$ (\ref{eq:1}) and
$e^{-t_2}\mathcal{G}=\mathcal{F}(p)$ ((\ref{eq:2}). Therefore we have
$i(\mathcal{F}(q),\mathcal{F}(p)) = e^{t_1 - t_2}$.
\begin{figure}[ht]
	\centering
  \includegraphics[scale=0.3]{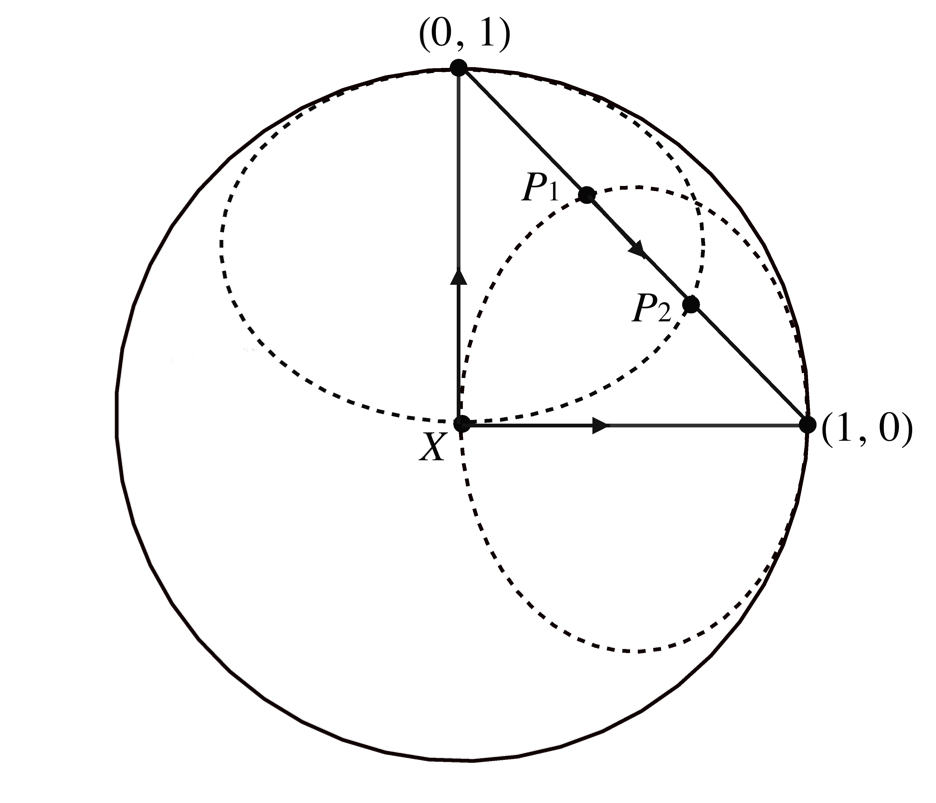}
  \caption{The real slice of $\mathbb{CH}^2\subset \mathbb{C}^2$
    coincides with the Klein model $\mathbb{RH}^2\subset \mathbb{R}^2$
    of the real hyperbolic plane of constant curvature $-1$. }
	\label{fig1}
\end{figure}
\begin{remark}
  A simple calculation shows that $t_2 -t_1= \log(2)$; hence,
  $i(\mathcal{F}(q),\mathcal{F}(p))= \frac{1}{2}$.
\end{remark}

\noindent\textbf{Step 3.} 
The holomorphic automorphism given by
$\phi (z,w)= (e^{-i \theta}z,w)$, for $(z,w)\in \mathbb{CH}^2$, is an
isometry of $\mathbb{CH}^2$ and sends the two horocycles
$H(\gamma_i,1)$ to the horocycles $H(\phi(\gamma_i),1)$, for
$i=1,2$. The Teichm\"uller geodesic ray $f(\phi(\gamma_1))$ is now
generated by $e^{i \theta} q$, whereas the Teichm\"uller geodesic ray
$f(\phi(\gamma_2))$ is still generated by $p \in Q(X)$. Since the
distance between $P_1$ and $P_2$ is equal to the distance between
$\phi(P_1)$ and $\phi(P_2)$, using Step 2 and the continuity of the
geometric intersection pairing we conclude that
$i(\mathcal{F}(e^{i \theta}q), \mathcal{G}) = \frac{1}{2}$ for all
$\theta \in \mathbb{R}/2\pi\mathbb{Z}$. However, this contradicts the
following Proposition~\ref{prop:intersection}.\qed
\restate[Proposition]{prop:intersection}{
  Let $q\in Q_1\T_{g,n}$ and $\mathcal{G}\in \mathcal{MF}_{g,n}$.
  There exist $v_1, \ldots, v_N \in \C^{*}$ such that
  $i(\mathcal{F}(e^{i\theta}q),\mathcal{G})=\sum_{i=1}^{N} |
  \text{Re}(e^{i\theta/2}v_i)|$ for all $\theta \in
  \mathbb{R}/2\pi\mathbb{Z}$.}
\begin{proof}[Proof of Proposition~\ref{prop:intersection}]
  Let $q \in Q(X)$ be a unit area quadratic differential. We assume
  first that $q$ has no poles and that $\mathcal{G}$ is an isotopy
  class of simple closed curves. The metric given by $|q|$ is flat
  with conical singularities of negative curvature at its set of zeros
  and hence the isotopy class of simple closed curves $\mathcal{G}$
  has a unique geodesic representative, which is a finite union of
  saddle connections of $q$. In particular, we can readily compute
  $i(\mathcal{F}(e^{i\theta}q),\mathcal{G})$ by integrating
  $|\text{Re} (\sqrt{e^{i\theta}q} )|$ along the union of these saddle
  connections. It follows that:
  \begin{equation}\label{eq:intersection}
    i(\mathcal{F}(e^{i\theta}q),\mathcal{G}) = \sum_{i=1}^{N} |
    \text{Re}(e^{i\theta/2}v_i)|
    \quad \text{for all} \quad \theta \in \mathbb{R}/2\pi\mathbb{Z}
  \end{equation}
  where $N$ denotes the number of the saddle connections and
  $\{ v_i \}_{i=1}^{N} \subset \C^{*}$ are their associated holonomy
  vectors.

  We note that when $q$ has simple poles, there need not be a geodesic
  representative in $\mathcal{G}$ anymore. Nevertheless, equation
  (\ref{eq:intersection}) is still true by applying the argument to a
  sequence of length minimizing representatives.

  Finally, we observe that the number of saddle connections $N$ is
  bounded from above by a constant that depends only on the topology
  of the surface. Combining this observation with the fact that any
  $\mathcal{G} \in \mathcal{MF}_{g,n}$ is a limit of simple closed
  curves and that the geometric intersection pairing $i(\cdot,\cdot):
  \mathcal{MF}_{g,n} \times \mathcal{MF}_{g,n} \rightarrow \R$ is
  continuous, we conclude that equation~(\ref{eq:intersection}) is
  true in general.
\end{proof}



\section{Final remarks}

We conclude this note with a few open questions and further results.~\\

\subsection*{Questions} 
~\\

\noindent\textbf{1.} Is Theorem~\ref{thm:disks:intro} true for
$f: \mathbb{CH}^1 \hookrightarrow \T_{g,n}$ a (real) $C^1$-smooth
\textit{local} isometry?~\\

\noindent\textbf{2.} Is there a \textit{round} complex two-dimensional
linear slice in $T_X\T_{g,n}$?~\\

\noindent\textbf{3.} Is there a holomorphic isometric immersion $f: (\mathcal{M},g) \hookrightarrow
\T_{g,n}$ from a Hermitian manifold with $\text{dim}_{\C}\mathcal{M} \geq 2$?~\\

\noindent\textbf{4.} Is there a holomorphic retraction
$\T_{g,n} \xtwoheadrightarrow{g} \mathbb{CH}^1$ onto every
Teichm\"uller disk $\mathbb{CH}^1 \xhookrightarrow{~~f~~} \T_{g,n}$
such that $g \circ f = id_{\mathbb{CH}^1}$? Equivalently, does the
Caratheodory metric equal to the Kobayashi metric for every complex
direction of $\T_{g,n}$?~\\

\subsection*{Further results}~\\

The following two theorems suggest that the answers to questions 2 \&
3 are \textit{no}.

\begin{thm}\label{thm:slice}
  There is no complex linear isometry
  $P: (\C^2,||\cdot||_2) \hookrightarrow (Q(X),||\cdot||_1)$.
\end{thm}
\begin{remark}
  This result is used in the proof of
  Theorem~\ref{thm:bsds:dual:intro}. See~\cite{Antonakoudis:birational}
  for a proof.
\end{remark}

As an application of Theorem~\ref{thm:bsds:intro}, we prove:
\begin{thm}\label{thm:kaehler}
  Let $(\mathcal{M},g)$ be a complete K\"ahler manifold with
  $\text{dim}_{\C}\mathcal{M} \geq 2$ and holomorphic sectional
  curvature at least $-4$. There is no holomorphic map
  $f:\mathcal{M}\rightarrow \T_{g,n}$ such that $df$ is an isometry
  on tangent spaces.
\end{thm}
\begin{proof}
  The monotonicity of holomorphic sectional curvature under
  holomorphic maps and the existence of (totally geodesic) holomorphic
  isometries $\mathbb{CH}^1 \hookrightarrow \T_{g,n}$ through every
  complex direction imply that $\mathcal{M}$ has constant holomorphic
  curvature -4.~\cite{Royden:metric} Since $\mathcal{M}$ is a complete
  K\"ahler manifold, we have $\mathcal{M}\cong\mathbb{CH}^{N}$, which
  is impossible when $N \geq 2$ by Theorem~\ref{thm:bsds:intro}.
\end{proof}

We also mention the following immediate corollaries of
Theorem~\ref{thm:slice} and Theorem~\ref{thm:kaehler}, respectively.
\begin{cor}\label{cor:hermitian}
  Let $(\mathcal{M},g)$ be a Hermitian manifold with
  $\text{dim}_{\C}\mathcal{M} \geq 2$. There is \textit{no}
  holomorphic isometric submersion
  $\displaystyle g: \T_{g,n} \twoheadrightarrow \mathcal{M}$. 
\end{cor}

\begin{cor}\label{cor:kaehler}
  There is no holomorphic, totally geodesic isometry from a K\"ahler
  manifold $\mathcal{M}$ into a Teichm\"uller space $\T_{g,n}$, so long
  as $\mathcal{M}$ has dimension two or more.
\end{cor}

For partial results and references towards question 4, see
~\cite{Kra:abelian},~\cite{McMullen:covs} and~\cite{Fletcher:Markovic:survey}.


\end{document}